\documentclass[12pt, reqno]{amsart}
\usepackage[active]{srcltx}
\usepackage[mathscr]{eucal}

\title[Gevrey regularity for quasi-geostrophic equation]{Gevrey regularity for the supercritical
 quasi-geostrophic equation}
\author{Animikh Biswas}
\address{Department of Mathematics \& Statistics, 
University of Maryland, Baltimore County,
1000 Hilltop Circle, Baltimore, MD - 21250, USA.}
\email{abiswas@umbc.edu}
\dedicatory{To my teacher Professor Ciprian Foias on the occasion of his eightieth birthday.}
 \thanks{This research was partly supported by NSF grant DMS11-09532}

\newtheorem{theorem}{Theorem}[section]
\newtheorem{lemma}[theorem]{Lemma}

\newtheorem{cor}[theorem]{Corollary}

\newtheorem{rem}{Remark}

\newcommand{\comments}[1]{}

\renewcommand{\phi}{\varphi}
\newcommand{\R}{\mathbb R}

\newcommand{\nn}{\nonumber}
\newcommand{\D}{\displaystyle }
\newcommand{\ka}{\kappa }
\newcommand{\Z}{\mathbb Z}
\newcommand{\cal}{\mathcal }

\newcommand{\cA}{\cal A}

\newcommand{\cf}{\cal F}

\newcommand{\ra}{\rightarrow}

\newcommand{\h}{\mathbb H}
\newcommand{\hh}{\dot{\h}}

\newcommand{\ds}{\frac{d}{ds}}
\renewcommand{\l}{\langle}
\renewcommand{\r}{\rangle}

\renewcommand{\dj}{\Delta_j}
\newcommand{\dk}{\Delta_k}
\newcommand{\tdk}{\widetilde{\Delta}_k}
\newcommand{\tdj}{\widetilde{\Delta}_j}
\newcommand{\sk}{S_k}
\newcommand{\sj}{S_j}
\newcommand{\es}{e^{\lambda\, {\D s^{\frac{\alpha}{\kappa}}\Lambda^{\alpha}}}}
\newcommand{\esx}{e^{\lambda {\D s^{\frac{\alpha}{\kappa}}|\xi|^{\alpha}}}}
\newcommand{\esxt}{e^{\lambda {\D s^{\frac{\alpha}{\kappa}}|\xi-\eta|^{\alpha}}}}
\newcommand{\estt}{e^{\lambda {\D s^{\frac{\alpha}{\kappa}}\tau^\alpha|\eta|^{\alpha}}}}
\newcommand{\esttt}{e^{\lambda {\D s^{\frac{\alpha}{\kappa}}|\eta|^{\alpha}}}}
\newcommand{\T}{\widetilde{\theta}}
\newcommand{\U}{\widetilde{u}}
\newcommand{\tn}{\T^{(n)}}
\newcommand{\un}{\U^{(n)}}
\newcommand{\tno}{\T^{(n+1)}}

\def\esssup{\text{ess sup}}

\def\intav#1{\mathchoice
          {\mathop{\vrule width 6pt height 3 pt depth -2.5pt
                  \kern -9pt \intop}\nolimits_{\kern -6pt#1}}%
          {\mathop{\vrule width 5pt height 3 pt depth -2.6pt
                  \kern -6pt \intop}\nolimits_{#1}}%
          {\mathop{\vrule width 5pt height 3 pt depth -2.6pt
                  \kern -6pt \intop}\nolimits_{#1}}%
          {\mathop{\vrule width 5pt height 3 pt depth -2.6pt
                  \kern -6pt \intop}\nolimits_{#1}}}

\newcommand{\charfn}[1]{{\raisebox{1.2pt}{\mbox{$\chi
_{\kern-1pt\lower3pt\hbox{{$\scriptstyle{#1}$}}}$}}}}

\date{}
\begin{document}
\begin{abstract}
In this paper, following the techniques of Foias and Temam,
we establish suitable Gevrey class regularity of solutions to the supercritical
  quasi-geostrophic equations in the whole space, with initial data in ``critical" Sobolev spaces. 
Moreover, the Gevrey class that we obtain is ``near optimal" and
as a corollary,
we  obtain temporal decay rates of higher order Sobolev norms of the solutions.
Unlike the Navier-Stokes  or the subcritical quasi-geostrophic equations,
the low dissipation poses a difficulty in establishing Gevrey regularity.
A new commutator estimate in Gevrey classes, involving the dyadic Littlewood-Paley operators, is 
established that allow us to exploit the cancellation properties of the equation and circumvent this difficulty.
\end{abstract}
\maketitle
\section{Introduction}
We consider
the dissipative, two-dimensional (surface) quasi-gesotrophic equation (henceforth, QG) on 
$\R^2 \times (0,\infty)$ given by 
\begin{gather}    \label{QG}
\left.
\begin{array}{l}
\partial_t \theta + \Lambda^\ka \theta - u\cdot \nabla \theta =0,
\theta(0)=\theta_0,\\
u = (-R_2\theta,R_1\theta), R_i = \partial_i\Lambda^{-1}, i=1,2.
\end{array}
\right\}
\end{gather}
Here $u$ is the velocity field, $\theta$ is the temperature, the operator $\Lambda = (-\Delta)^{1/2}$
with $\Delta $ denoting the Laplacian and
the operators $R_i$  are the usual Riesz transforms. The cases $\kappa >1, \kappa =1$ and $0<\kappa <1$ are known as the subcritical, critical and supercritical cases respectively.
The  QG arises in geophysics and meteorology (see, for instance \cite{const,constmajtab, Pedlosky}).
Moreover, the critical QG is the dimensional analogue of the three dimensional Navier-Stokes equations.
 This equation has received considerable attention recently; 
see \cite{resnick, cwu, cordoba2004, Chen},
 and the references therein. 

In this paper, we establish higher order (Gevrey class) regularity of solutions to the  supercritical QG in the whole space $\R^2$, with initial data in the critical Sobolev space $\h^{2-\kappa}$. Although such results are known for the subcritical and critical QG \cite{Dong, Dong2},  to the best of our knowledge this is the first such result for the supercritical case. We
follow the Gevrey class technique introduced in  the seminal work of Foias and Temam \cite{ft} for the Navier-Stokes equations, where they established analyticity, and provided explicit estimates of the  analyticity radius, of solutions to the Navier-Stokes equations. In their 
 approach, one avoids cumbersome recursive estimation of higher order derivatives and intricate combinatorial arguments. 
 Since its introduction, 
 Gevrey class technique has become a standard
 tool for studying analyticity properties of solutions for a wide class of dissipative equations 
and in various functional spaces
 (see \cite{GK, cao1999navier, ferrari1998gevrey, biswas2007existence, biswas2010navier, biswas2012} and the references therein). 
 In \cite{oliver2000remark}, and subsequently in \cite{biswas}, it was shown how Gevrey norm estimates can be used  to derive sharp bounds for the (time) decay of higher order
derivatives of solutions to a wide class of dissipative equations including the Navier-Stokes equations. 
Other approaches to analyticity and higher order regularity can be found in \cite{pav, miura2006, guber} for the 3D NSE and \cite{Dong2} for the subcritical surface quasi-geostrophic equation. 

The subcritical QG is fairly well understood; it possesses a globally regular (even, analytic) solution 
for adequate initial data (see \cite{cwu, Carrillo, Dong2}), as well as 
a compact global attractor \cite{ju2005}. However, relatively less is known concerning
 the critical ($\kappa=1$) and
 the supercritical ($\kappa <1$) QG.  Concerning the critical case, 
the authors in \cite{ccw} proved the existence of unique, globally regular solution for small initial data in
$L^\infty$. The global well-posedness for the critical QG,  for initial data
of arbitrary size, was solved independently in recent works \cite{Caffarelli, Kiselev}; 
 subsequently they were generalized to include initial data in larger functional spaces  
\cite{Abidi, Wang}. In \cite{constvicol}), an  alternative proof of global well-posedness 
was found using a ``nonlocal maximum principle".
The supercritical case also received considerable 
attention of late, although less is known about it. It has been shown that it is {\em locally} well posed for initial data of arbitrary size in appropriate functional spaces while being globally well-posed for sufficiently {\em small initial data } in adequate functional spaces 
(see \cite{Chae, miura, Ju2004, Ju2006, Chen, Hmidi, Dong3} and the references therein). 
Although the global well posedness for arbitrary initial data is still open (as of this writing) for the supercrtical  QG, a regularity criterion for solutions has been establsished \cite{cwuholder} and \cite{Dong3} and eventual regularity has been addressed recently in \cite{dabkowsky}.

Recall that $f$ is said to belong to $L^2-$based Gevrey class $G_\alpha$ if 
\begin{gather}   \label{Gevreyclass}
\|f\|_{\h^{n}} \le C_f \left(\frac{n!}{\rho^n}\right)^{\frac{1}{\alpha}},
\end{gather}
where $\h^n$ denotes the usual Sobolev space of order $n$.
This can be characterized by the finiteness of the exponential norm 
$\|e^{\rho' \Lambda^{\alpha}}f\|_{L^2}$ for all $\rho' < \rho$. 
If $\alpha=1$, then $f$ is analytic with (uniform) analyticity radius $\rho$
(see \cite{levermore, oliver2000remark}) while $\alpha<1$ corresponds to sub-analytic Gevrey classes. 
We show that for the supercritical QG in the whole space $\R^2$,
and for sufficiently small initial data  in  $\h^{2-\kappa}$, a solution to \eqref{QG}
exists which moreover satisfies satisfying $\sup_{t>0}\|e^{\rho(t)\Lambda^{\alpha}}\theta(t)\|_{\hh^{2-\kappa}} < \infty $. Here $\alpha < \kappa \le 1$ and $\rho(\cdot)$ is an adequate function. 
This immediately implies a higher order decay estimate similar to \eqref{Gevreyclass}. 
The result also holds locally for arbitrary initial data. As noted in Remark \ref{rem:optimal}, 
the  Gevrey class we obtain is ``near optimal" and moreover,
 our result includes as corollary the higher regularity and decay results established in \cite{Dong, Dong3} with sharper constants (see Remark \ref{rem:dongcomparison}).

\comments{Unlike in \cite{ft}, where they considered the periodic case for the Navier-Stokes equations, 
we focus on the whole space setting, although 
our method applies equally well to the space periodic case.
It should be noted that whether or not a solution to the critical QG is analytic, was left as an open problem in \cite{Dong}. We showed in \cite{biswas} that for small initial data in a subclass of $L^\infty$ (more precisely, those with $\|\cf \theta\|_{L^1}$ sufficiently small, where $\cf$ denotes the Fourier transform), the solution is globally analytic. However, there we did not have a corresponding analyticity result for arbitrary initial data. Here, we show that for initial data in the aforementioned Sobolev space, the solutions to the critical and supercritical QG (constructed for instance in \cite{miura}) belong to suitable (subanalytic) Gevrey classes; }

The idea of the proof follows that of \cite{weis1980} and \cite{fujita1964navier} for the Navier-Stokes equations, suitably modified for Gevrey classes. A crucial step in establishing Gevrey regularity for the Navier-Stokes or the subcritical QG (see \cite{biswas2010navier, biswas}) is to obtain an estimate in Gevrey classes
of the form (in 2D)
\[
\|e^{\lambda \Lambda^{\alpha}}(f g)\|_{\hh^{\zeta}} \le C \|e^{\lambda \Lambda^{\alpha}}(f )\|_{\hh^{\zeta_1}}\|e^{\lambda \Lambda^{\alpha}}(g)\|_{\hh^{\zeta_2}}, \  \zeta=\zeta_1+\zeta_2-1,
\zeta_1,\zeta_2  < 1, \zeta_1+\zeta_2 >0,
\]
where $\hh^{\zeta}$ denotes homogeneous Sobolev spaces.
This can be derived from the corresponding inequality in Sobolev spaces, namely,
\[
\|f g\|_{\hh^{\zeta}} \le C \|f \|_{\hh^{\zeta_1}}\|g\|_{\hh^{\zeta_2}}, 
\zeta=\zeta_1+\zeta_2-1, \zeta_1,\zeta_2  < 1, \zeta_1+\zeta_2 >0.
\]
However,
 unlike the Navier-Stokes equations or the subcritical QG, due to the the low dissipation in the  supercritical case, one has to work in higher regularity spaces, namely $\h^{\delta}, \delta >1$, for well-posedness. An inequality of the above type does not hold in general in such spaces (as $\zeta_1,\zeta_2$ will have to be taken larger than $1$). This  poses a hurdle in establishing Gevrey regularity. To get around this difficulty, we establish in Theorem \ref{thm:commutest}
  a new   (to the best of our knowledge)  commutator estimate in Gevrey classes involving the dyadic Littlewood-Paley operators, which may be of independent interest as well.
This  commutator estimate allows us to exploit the cancellation properties of the equation to get around 
the  challenges posed by low dissipation.
 Usually,  if one works in Gevrey classes, one looses the cancellation properties of the equation which is available in $L^2$  spaces; see however the work in \cite{levermore, kukvic} where a certain cancellation property  in the analytic Gevrey class 
was also used to establish analyticity estimate for the space periodic Euler equation.
 Our technique can be generalized to initial data in critical (and noncritical) Besov spaces. Due to the lack of Hilbert space structure as well as the Plancherel theorem, the corresponding estimates are more involved and will be presented in a future work.

The organization of the paper is as follows. In Section 2, we state our main results; in Section 3, we develop the requisite notation and background material while Sections 4 and 5 are devoted to the proof of the main results.

\section{Main Results}  \label{sec:2}
 Denote by $\Delta $ the Laplacian and by 
$\Lambda =(-\Delta)^{1/2}$. For notational parsimony, we will denote $\|f\|_{L^2}=\|f\|$.
The Sobolev and the homogeneous Sobolev spaces on $\R^2$ are respectively denoted by $\h^m$ and $\hh^m, m \in \R$. 
Recall that the corresponding norms are given by 
\begin{gather*}
\|f\|_{\h^m} = \|(I+\Lambda)^mf\|\ \mbox{and}\ \|f\|_{\hh^m}=\|\Lambda^mf\|,
\quad m \in \R.
\end{gather*}
Recall that by Plancherel theorem,
\begin{gather*}
\|f\|_{\h^m}=\left(\int (1+|\xi|)^{2m}|\cf f(\xi)|^2\, d\xi\right)^{1/2}\ \mbox{and}\
\|f\|_{\hh^m}=\left(\int |\xi|^{2m}|\cf f(\xi)|^2\, d\xi\right)^{1/2};
\end{gather*}
here, and henceforth, $\cf$ denotes the Fourier transform.
As is well known, the Sobolev spaces are Hilbert spaces for all $m$. The homogeneous Sobolev spaces on the other hand are Hilbert spaces for $m <1$, while they are normed inner product spaces (but not complete) for all $m \ge 1$ (see \cite{danchinfourier} , \cite{bahourifourier}).

{\bf Gevrey Norms:}
Let $0\le \alpha \le 1$. We denote the Gevrey norms by
\begin{gather*} 
\|f\|_{G(s)}=\|\es f\| \mbox{and}\  
\|f\|_{G(s), \hh^m}= \|\es f\|_{\hh^m},\quad (\lambda >0\ \mbox{fixed}\,).
\end{gather*}
The Gevrey norms are characterized by the decay rates of higher order derivatives, namely, 
if $\|f\|_{G(s), \hh^m} < \infty $ for some $m \in \R$, then we have the higher derivative estimates 
\begin{gather}  \label{gevreydecay}
\|f\|_{\hh^{m+n}} \le \left(\frac{n!}{\rho^n}\right)^{\frac{1}{\alpha}}\|f\|_{G(s), \hh^m}\ 
\mbox{where}\ \rho = \lambda \alpha s^{\frac{\alpha}{\kappa}}\ \mbox{and}\ n \, \in \, {\mathbb N}.
\end{gather}
In particular, when $\alpha =1$,  $f$ in (\ref{gevreydecay}) is  analytic with (uniform) analyticity radius $\rho$,
while for $\alpha < 1$ the corresponding functional classes are referred to as subanalytic gevrey classes.
For  the above mentioned facts including  (\ref{gevreydecay}), see Theorem 4 in \cite{levermore} and Theorem 5 in \cite{oliver2000remark}.

\begin{rem}
{\em 
The indices of $s$ appearing in the definition of the Gevrey norms and in  inequality (\ref{maincommutest}) below, 
allow us to establish global Gevrey regularity result for small data in the whole space. They are dictated by the scaling properties of the equation. It is not possible (at least, in our work) to establish such global results unless they are
 precisley of that form. The global Gevrey regularity result in turn enables us to establish decay result for higher derivatives as given in Corollary \ref{cor:decay} below.
}
\end{rem}

We now describe  our main results.
The first one is a commutator estimate involving Gevrey norms which may be of independent interest. The second, which employs the first in its proof, concerns Gevrey regularity of solutions of the critical and subcritical quasi-geostrophic equations.
\subsection{A Commutator estimate in Gevrey classes}
The commutator of two operators is defined as 
\begin{gather*}
[A,B]=AB-BA. 
\end{gather*}
The estimate for the commutator $[f,\es \dj]$, where $\Delta_j$ denotes the (homogeneous) dyadic Littlewood-Paley operator, 
 is crucial for our work. 
\begin{theorem}  \label{thm:commutest}
\comments{Let $0\le \delta_1, \delta_2 <1, \delta_1+\delta_2>0$ and}
 Let $f,g \in L^2$ with $\es f \in \hh^{1+\delta_1}, \es g \in \hh^{\delta_2}$ and 
\begin{gather}  \label{paramrestr}
\min\{\zeta , \delta_1, \delta_2 \}> 0, \delta_1+\zeta <1,  \delta_2 <1\ \mbox{and}\ \zeta < \alpha.
\end{gather}
Denote by $\dj, j \in \Z $ the dyadic (homogeneous)  Littlewood-Paley  operators.
There exists a constant $C$, independent of
$j, s, f$ and $g$, and a sequence of constants $\{c_j\}_{j\in \Z}$ (which may depend, in addition, on 
$ s, f$ and $g$) satisfying
$c_j \ge 0$ and $\sum_j c_j^2 \le 1$, such that
\begin{align}  
& \|[f,\es \dj]g\| \le  \nn \\
& Cc_j\| \es f\|_{\hh^{1+\delta_1}}\| \es g\|_{\hh^{\delta_2}}
\left\{s^{\frac{(\alpha-\zeta)}{\kappa}} 2^{-(\delta_1+\delta_2+\zeta-\alpha)j} 
+  2^{-(\delta_1+\delta_2)j} 
\right\}. \label{maincommutest}
\end{align}
\end{theorem}
For definition of $\dj, \sj$ in  Theorem \ref{thm:commutest}, see Section 3 below.
\subsection{Gevrey Regularity for the quasi-geostrophic equations} 
Here we will consider only the critical and super-critical cases, i.e., $0<\kappa \le 1$. 
For $\beta>0$ and a measurable function $\Theta:(0,T) \ra \hh^{2-\kappa+\beta} $, we denote
\begin{gather}  \label{pathnormdef}
\|\Theta(\cdot)\|_{E_T}:= 
= \esssup_{0<s<T}\,s^{\frac{\beta}{\kappa}}\|\es \Theta(s)\|_{\hh^{2-\kappa+\beta}},
\end{gather}
provided the right hand side is finite.
\begin{theorem}  \label{thm:main}
Let $ \kappa \le 1, \alpha <\kappa $ and $\theta_0 \in \h^{2-\kappa}$. 
There exist $\beta>0$ and $T>0$ and a solution $\theta(\cdot)$ on $[0,T]$ of  (\ref{QG})  such that 
\begin{gather*}
\|\theta(\cdot)\|_{E_T}\comments{:= \sup_{0<t<T}\max\{t^{\frac{\beta}{\kappa}}\|\theta(t)\|_{G(t),\hh^{2-\kappa+\beta}}, \|\theta(t)\|_{G(t),\hh^{2-\kappa}} \}}\le C\|\theta_0\|_{\h^{2-\kappa}},
\end{gather*}
where the constant $C$ is independent of $\theta_0$ and $T$. Furthermore,
 in case $\|\theta_0\|_{\hh^{2-\kappa}}$ is adequately small, 
we can take $T= \infty$.
\end{theorem}
\begin{rem}
{\em 
In case $\theta_0 \in \h^{2-\kappa +\epsilon}$ with $ \epsilon >0$, following the method presented here, we can provide an explicit estimate of $T$ in Theorem \ref{thm:main} above  in terms of $\|\theta_0\|_{\h^{2-\kappa+\epsilon}}$. However, in the critical space $\h^{2-\kappa}$ considered here, $T$ depends in a more complicated way on $\theta_0$, not just on its norm.}
\end{rem}
\begin{rem}   \label{rem:optimal}
{\em The Gevrey regularity result presented in Theorem \ref{thm:main} is ``near optimal" since the solution
in the linear case (i.e., when the nonlinearity in the quasi-geostrophic equation is not present)  belongs to the same Gevrey class with $\alpha=\kappa$, and no better. Though our result for the supercritical case is new, 
for the critical case $\kappa=1$, 
Theorem \ref{thm:main}  shows that for arbitrary initial data $\theta_0 \in \h^{1}$, 
the solutions are in all subanalytic Gevrey classes. 
In \cite{biswas} we showed that the solution to the critical quasi-geostrophic equation
is analytic if $\|\cf \theta_0\|_{L^1}$ is sufficiently small; see also \cite{shterenberg} for a similar result in fractal burgers equation. Thus,
 it would be interesting to see if one can obtain the optimal Gevrey class regularity (i.e., $\alpha = \kappa$ in Theorem \ref{thm:main}) for initial data in the critical Sobolev space $\h^{2-\kappa}$. 
}
\end{rem}
\begin{rem}
{\em Following our proof, it is not difficult to show that in Theorem \ref{thm:main}, the functiom $s \ra \es \theta(s)$ in fact belongs
to $C([0,T]; \h^{2-\kappa})$. Moreover,
the definition of the norm $\|\theta(\cdot)\|_{E_T}$ can be modified to
\begin{gather*}
\|\theta(\cdot)\|_{\widetilde{E_T}}:= \sup_{0<t<T}\max\{t^{\frac{\beta}{\kappa}}\|\theta(t)\|_{G(t),\hh^{2-\kappa+\beta}}, \|\theta(t)\|_{G(t),\hh^{2-\kappa}} \}.
\end{gather*}
The conclusions of Theorem \ref{thm:main} still hold with 
$\|\theta(\cdot)\|_{\widetilde{E_T}}$ in place of $\|\theta(\cdot)\|_{E_T}$. This method is inspired by the work of \cite{weis1980} and \cite{fujita1964navier} in case of the Navier-Stokes equations,
 where a higher order regularity gain due to dissipation
is used to control the critical norm.

}
\end{rem}
\begin{cor}   \label{cor:decay}
For any $n\,  \in \, {\mathbb N}$ satisfying $n > 2-\kappa$ and $\alpha < \kappa$, 
for some constant $C$ independent of $n$ and $s$,
the solution $\theta (\cdot)$ in Theorem \ref{thm:main} above satisfies the higher order decay estimates
\begin{gather}   \label{solutiondecay}
\|\Lambda^n \theta(s)\|_{\hh^{2-\kappa}} \le C\|\theta_0\|
\frac{(n!)^{\frac{1}{\alpha}}}{\rho^{\frac{n}{\alpha}}}\ 
\mbox{where}\  
\rho = \lambda \alpha s^{\frac{\alpha}{\kappa}}
\ \mbox{and}\ s\, \in\, (0,T).
\end{gather}
\end{cor}
The proof of the corollary follows immediately from Theorem \ref{thm:main} and (\ref{gevreydecay}).
\begin{rem}  \label{rem:dongcomparison}
{\em In \cite{Dong}, it was proven that in case $\|\theta_0\|_{\hh^{2-\kappa}}$ is sufficiently small,
there exists constants $C_n$ such that the decay estimate 
\begin{gather*}
\|\Lambda^n \theta(s)\|_{\hh^{2-\kappa}} \le \frac{C_n}{s^{\frac{n}{\kappa}}}
\end{gather*}
holds. This is the same rate as in (\ref{solutiondecay}) except that the constants $C_n$ were not identified there,
which in our case follows as a consequence of Gevrey regularity. Moreover, the constants in (\ref{solutiondecay})
are ``near optimal" since the (optimal) rate for the linear case is same as in (\ref{solutiondecay}) with $\alpha =\kappa$
(see Remark \ref{rem:optimal}).

}
\end{rem}

\section{Notation and Preliminaries}

We will need the following notions and some standard results from harmonic analysis to proceed.
For more details, see for instance \cite{Chemin1}, \cite{danchinfourier}, \cite{bahourifourier}, or \cite{Runst}.
\subsection{Littlewood-Paley Decomposition}
Let $\phi, \psi\, \in {\cal D} (\R^d)$, with ranges contained in the interval $[0,1]$, 
and such that 
\begin{gather*}
\psi(\xi)=
\left\{
\begin{array}{l}
1, |\xi| \le \frac12,\\
0, |\xi \ge 1
\end{array}
\right.
\ \mbox{and}\ 
\phi(\xi)= \psi(\xi/2)-\psi(\xi).
\end{gather*}
Let  $\Delta_j$ and $\sj$ be the (homogeneous) dyadic Littlewood-Paley projections given by
\begin{gather*}
\cf (\dj f) = \phi(\cdot/2^j)\cf f\ \mbox{and}\  \cf(\sj f)=\psi(\cdot/2^{j-3})\cf f.
\end{gather*}
We denote the (open) ball $B(r)$ and the (open) annulus $\cA (r_1,r_2), 0<r_1<r_2$ by
\begin{gather*}
B(r)=\{\xi:|\xi| < r\}\ \mbox{and} \
\cA (r_1,r_2)=\{\xi:r_1< |\xi| < r_2\}.
\end{gather*}
For each $j \in \Z$, 
the Fourier spectrum of  $\dj f$ (respectively, $\sj f$) is ``localized" in  $\cA(2^{j-1},2^{j+1})$ 
(respectively,  $B(2^{j-3})$), i.e.,
\[
\cf (\dj f)=0\ \mbox{for}\ \xi \, \in \,\cA(2^{j-1},2^{j+1})^c\ \mbox{and}\ 
\cf (\sj f)(\xi)=0\  \mbox{for}\ \xi \, \in \,B(2^{j-3})^c.
\]
Moreover, 
\begin{gather*}
S_j = \sum_{k \le j-4}\Delta_k,
\end{gather*}
where the equality holds in the space of distributions ``modulo polynomials" \cite{bahourifourier, danchinfourier}.
The (homogeneous) paraproduct formula is given by
\begin{gather}  \label{paraprod}
fg = T_fg + T_gf +R(f,g),
\end{gather}  
where, denoting $\tdj = \sum_{k=j-3}^{j+3}\dk g$, $T_fg$ and $R(f,g)$ are given by
\begin{gather*}
T_fg = \sum_{j\in \Z} S_jf\Delta_jg,\ R(f,g)=\sum_{j,k:|j-k| \le 3} \Delta_jf \Delta_kg
=\sum_j \Delta_j f \tdj g, 
\end{gather*}
The following facts concerning the paraproduct decomposition will be used throughout:
\begin{gather} \label{spectralloc}
\Delta_i \Delta_k =0\ \mbox{if}\ |i-k| \ge 2,\ 
\Delta_i(S_kf\Delta_kg)=0\ \mbox{if}\ |i-k|\ge 3
\end{gather}
and
\begin{gather} \label{spectralloc2}
\Delta_i (\Delta_k f \tdk g) =0\ \mbox{if}\ i \ge k+6.
\end{gather}
The above two facts, namely (\ref{spectralloc}) and (\ref{spectralloc2}), follow readily from
the spectral localization of the operators $\dj$ and $\sj$.
\subsection{Bernstein and other related inequalities}
Let $f$ and $g$ be two Schwartz class functions with Fourier spectrum localized in the ball $B(r\mu)$
and annulus $\cA (r_1\mu,r_2\mu)$ respectively, with $\min\{r,r_1,r_2, \mu\} >0$. Then, for some constants $C,C_1,C_2$, depending
only on $r,r_1, r_2$, 
we have
\begin{gather} \label{bernstein1}
\|f\|_{\hh^{m}} \le C\mu^{m}\|f\| \ (m >0)\ \mbox{and}\ C_1\mu^{m'}\|g\|\le \|g\|_{\hh^{m'}} 
\leq C_2\mu^{m'}\|g\|,\ (m' \in \R).
\end{gather}
Moreover, 
\begin{gather}  \label{bernstein2}
\|S_jf\|_{\hh^m} \le \|f\|_{\hh^m}\ \mbox{and}\ 
C_1\|f\|\le \left(\sum_j \|\dj f\|^2\right)^{1/2} \le C_2\|f\|.
\end{gather}

As a consequence of the Young's convolution inequality and Parseval equality, we have
\begin{gather}  \label{youngs}
\|fg\|_{L^2} \le C\|\cf f\|_{L^p}\|\cf g\|_{L^q},\ \frac 32 = \frac 1p + \frac 1q, 1 \le p, q \le 2.
\end{gather}
We will also need  versions of the Bernsten's inequalities in the Fourier space that are easy to prove, 
namely, 
\begin{align}  
& C_1 2^{\alpha j}\|\cf \dj f\|_{L^p} \le \|\cf \dj \Lambda^\alpha f\|_{L^p}
\le C_2 2^{\alpha j}\|\cf \dj f\|_{L^p},\nn \\ 
& \|\cf \dj f\|_{L^p} \le C2^{2j(\frac 1p - \frac 1q )}\|\cf \dj f\|_{L^q},
 \  j \in \Z,\, 1 \le p \le q \le \infty,  \label{bernsteinfourier}
\end{align}
and where the constants $C, C_1, C_2$ are independent of $f$ and $j$
and the space dimension is two.

Recall that for $0\le \alpha \le 1$, the Gevrey norms that we will use were defined in Section \ref{sec:2} by
\begin{gather*} 
\|f\|_{G(s)}=\|\es f\| \mbox{and}\  
\|f\|_{G(s), \hh^m}= \|\es f\|_{\hh^m},\quad (\lambda >0\ \mbox{fixed}\,).
\end{gather*}
It is clear from the definition of the Gevrey norms that
\begin{gather*}
\|f\|_{G(s)} \ge \|f\|\ \mbox{and}\ \|f\|_{G(s), \hh^m}\ge \| f\|_{\hh^m}.
\end{gather*}
Moreover,
since $\es, \dj, \sj$ are all Fourier multipliers, they commute with each other, i.e.,
\begin{gather*}
\es \dj f = \dj \es  f\ \mbox{and} \ \es \sj f = \sj \es f,
\end{gather*}
 for $f$ in appropriate functional classes.
We will use these facts throughout without any further mention. 
The following inequalities will also be crucial. 

{\em 
Let $s >0$ and $\zeta_1, \zeta_2 \in \R$ satisfy
\begin{gather*}
\zeta_1+ \zeta_2 >0, \max\{\zeta_1,\zeta_2\} <1.
\end{gather*}
Then, for functions $f$ and $g$ belonging to $\hh^{\zeta_1}$ and $\hh^{\zeta_2}$ respectively, there exists a constant $C=C(\zeta_1,\zeta_2)$, which is independent of $s$, such that
\begin{align}
& \|f g\|_{\hh^{\zeta_1+\zeta_2-1}} \le C(\zeta_1,\zeta_2)\|f\|_{\hh^{\zeta_1}}\|g\|_{\hh^{\zeta_2}} \label{convineq1} \\
& \|f g\|_{G(s),_{\hh^{\zeta_1+\zeta_2-1}}} \le C(\zeta_1,\zeta_2)\|f\|_{G(s), \hh^{\zeta_1}}\|g\|_{G(s),\hh^{\zeta_2}}. \label{convineq2}
\end{align}
}
The first one 
is well known and is a consequence of a more general convolution inequality
of Kerman \cite{kerman} or can be found in \cite{Runst}.  
The second can easily be derived from the first as follows.  
Note first that for $\xi , \eta \in \R^2$ and $s \ge 0$, we have
 $|\xi| \le |\eta| + |\xi - \eta|$; consequently from the elementary inequality
\begin{gather}   \label{elemconvexity}
(x+y)^{\alpha} \le x^{\alpha}+y^{\alpha}, x,y \ge 0, 0 < \alpha \le 1, 
\end{gather}
we have
\begin{gather}  \label{convexity}
\esx \le \esttt \esxt.
\end{gather}
For notational simplicity, denote $\delta = \zeta_1+\zeta_2-1$, 
Using the Plancherel theorem and (\ref{convexity}),
\begin{align*}
& \|\es (fg)\|_{\hh^\delta}^2 = 
\int \left(\esx |\xi|^{\delta}\int f(\xi -\eta)g(\eta)\, d\eta\right)^2 d\xi \\
& \le \int \left(|\xi|^{\delta}\int \esxt |f(\xi -\eta)|\, \esttt|g(\eta)|\, d\eta\right)^2 d\xi \\
& \le \|\es f\|_{\hh^{\zeta_1}}^2\|\es g\|_{\hh^{\zeta_2}}^2.
\end{align*}
The last inequality follows by applying (\ref{convineq1}) to the ``auxilliary" functions
$\tilde{f}$ and $\tilde{g}$, where 
\begin{gather*}
\cf \tilde{f}(\xi) = \esx |f(\xi)|\ \mbox{and}\ \cf \tilde{g}(\xi) = \esx |g(\xi)|.
\end{gather*}
We have also used the elementary (yet, crucial) fact that 
\begin{gather*}
\|\es f\|_{\hh^{\zeta_1}}= \|\tilde{f}\|_{\hh^{\zeta_1}}\ \mbox{and}\ 
\|\es g\|_{\hh^{\zeta_2}}= \|\tilde{g}\|_{\hh^{\zeta_2}}.
\end{gather*}
This finishes the proof of \eqref{convineq1} and  \eqref{convineq2}.

Let $1\le p,q,r \le \infty$ with $1 + \frac 1r = \frac 1p + \frac 1q$.
Proceeding in an analogous manner and using  Young's convolution inequality,
 one readily obtains also the inequality
\begin{gather}  \label{gevreyconvolution}
\|\cf \left(\es (fg)\right)\|_{L^r}\le \|\cf (\es f)\|_{L^p}\|\cf (\es g)\|_{L^q}.
\end{gather}

\section{Proof of Theorem \ref{thm:commutest}}
In order to prove this result, we will need the following lemma, which may be regarded
as central to the proof of Theorem \ref{thm:commutest}.
\begin{lemma}  \label{lem:firstcommutest}
In the notation and setting of Theorem \ref{thm:commutest}, we have the estimate
\begin{align}
\|[T_f,\es\dj]g\| & \le Cc_j\left\{s^{\frac{(\alpha-\zeta)}{\kappa}} 2^{-(\delta_1+\delta_2+\zeta-\alpha)j} \|\sk \es f\|_{\hh^{1+\delta_1}}\| \es g\|_{\hh^{\delta_2}} \right. \nn \\
& \qquad \left.+  2^{-(\delta_1+\delta_2)j} \|\sk \es f\|_{\hh^{1+\delta_1}}\| \es g\|_{\hh^{\delta_2}}\right\}. \label{firstcommutest}
\end{align}
\end{lemma}
\begin{proof}
Due to (\ref{spectralloc}), we have 
\begin{gather*}
[T_f,\es\dj]g = \sum_{k:|k-j| \le 2}[\sk f,\es \dj]\dk g.
\end{gather*}
Note that 
\begin{align*}
\lefteqn{-\cf [\sk f,\es\dj]\dk g (\xi) =}\\
& \ \int \cf \sk f(\eta)\cf \dk g(\xi-\eta)\esx\left[\phi(\frac{\xi}{2^j})-\phi(\frac{\xi-\eta}{2^j})\right]d\eta \\
& \qquad \qquad 
+ \int \cf \sk f(\eta)\cf \dk g(\xi-\eta)\phi(\frac{\xi-\eta}{2^j})\left[\esx - \esxt \right]d\eta \\
& = I + II.
\end{align*}
Recall that from (\ref{spectralloc}), we have
\begin{gather*}  
I=II=0\ \mbox{if}\ \xi \in [\cA (2^{k-2},2^{k+2})]^c
\ \mbox{and}\ [\sk f,\es\dj]\dk g=0\ \mbox{if}\ |j-k| \ge 3.
\end{gather*}
Due to this, for any $\delta \in \R$ and $k \in [j-2,j+2] \cap \Z$, we have
\begin{gather}  \label{spectralloc1}
\|I\|\le C2^{-\delta k}\|\Lambda^\delta I\|\le C2^{-\delta j}\|\Lambda^\delta I\|
\ \mbox{and}\ \|II\|\le C2^{-\delta k}\|\Lambda^\delta II\|\le C2^{-\delta j}\|\Lambda^\delta II\|,
\end{gather}
where the constant $C$ is independent of $j, f$ and $g$. Now note that
since $\phi $ and all its derivatives are uniformly bounded, applying the mean value theorem (to $\phi$),
we obtain
\begin{gather*}   
|\phi(\frac{\xi}{2^j})-\phi(\frac{\xi-\eta}{2^j})| \le C_\phi 2^{-j}|\eta|, C_\phi=\|\phi'\|_{L^\infty}.
\end{gather*}
Inserting the above  estimate and (\ref{convexity}) in $I$ and using  
(\ref{spectralloc1}), we obtain
\begin{align*}
& \left|\int \cf \sk f(\eta)\cf \dk g(\xi-\eta)\esx\left[\phi(\frac{\xi}{2^j})-\phi(\frac{\xi-\eta}{2^j})\right]d\eta\right|
\\
& \le 
C2^{-(\delta_1+\delta_2)j}\left\|\Lambda^{(\delta_1+\delta_2-1)}
\int|\eta| |\esttt \cf \sk f(\eta)||\esxt\cf \dk g(\xi-\eta)|d\eta\right\|.
\end{align*}
Now using (\ref{convineq1})  with 
$\zeta_1=\delta_1, \zeta_2=\delta_2$, 
we finally obtain the estimate
\begin{align} 
 \|I\| & \le C2^{-(\delta_1+\delta_2)j} \|\sk  \es f\|_{\hh^{1+\delta_1}}\|\dk \es g\|_{\hh^{\delta_2}} \nn \\
& \le  C2^{-(\delta_1+\delta_2)j}c_j\|\es f\|_{\hh^{1+\delta_1}}\| \es g\|_{\hh^{\delta_2}}, \nn \\
& \mbox{where}\ 0<\delta_1, \delta_2<1\ \mbox{and}\ c_j = 
\frac{\left(\sum_{k=j-3}^{j+3}\|\dk \es g\|^2_{\hh^{\delta_2}}\right)^{1/2}}{C'\|\es g\|_{\hh^{\delta_2}}} ,  \label{Iest}
\end{align}
where we have also used the first inequality in (\ref{bernstein2}) above. Furthermore, the constant
 $C'$ in (\ref{Iest})  may depend only on $\delta_1,\delta_2$ and $\sum c_j^2 \le 1$. 
These facts  follow from the second inequality in (\ref{bernstein2}).

We will now estimate $II$.
Note that
\begin{align}
& \left|\esx - \esxt\right| =
\left| \int_0^1 \frac{d}{d\tau} ( e^{\lambda {\D s^{\frac{\alpha}{\kappa}} |\xi-(1-\tau)\eta|^\alpha}})d\tau\right| \nn \\
& \le C\lambda s^{\frac{\alpha}{\kappa}} \int_0^1 \frac{|\eta|}{|\xi-(1-\tau)\eta|^{1-\alpha}}
e^{\lambda {\D s^{\frac{\alpha}{\kappa}} |\xi-(1-\tau)\eta|^\alpha}}d\tau.
\end{align}
Since $0<\alpha \le 1$, from (\ref{elemconvexity}) it follows that
\begin{gather*}
|\xi - (1-\tau)\eta|^\alpha =
|(\xi-\eta) + \tau \eta|^\alpha \le |\xi-\eta|^\alpha + \tau^\alpha|\eta|^\alpha.
\end{gather*}
Consequently,
\begin{gather}  \label{expterm}
e^{\lambda {\D s^{\frac{\alpha}{\kappa}} |\xi-(1-\tau)\eta|^\alpha}} \le \esxt \estt.
\end{gather}
Recall that the support of $\cf \sk f$ is in $B(2^{k-3})$ and the support of 
$\cf \dk g$ is in $\cA (2^{k-1},2^{k+1})$. Thus, for the integrand in $II$ to be nonzero, 
we must have
\begin{gather}  \label{crucialtriangle}
|\xi| \ge |\xi-\eta| - |\eta| \ge 2^{k-1}-2^{k-3}=3(2^{k-3})\ge 3|\eta|.
\end{gather}
Since $0\le(1-\tau) \le 1$, this immediately implies that
\begin{gather}  \label{geometry}
|\xi -(1-\tau)\eta| \ge |\xi| - (1-\tau)|\eta| \ge |\xi|-|\eta| \ge \frac{2}{3}|\xi|.
\end{gather}
\comments{
For notational simplicity, temporarily denote $\tau_1 = 1-\tau, \eta_1 =-\eta$. Then $\xi-(1-\tau)\eta =\xi+\tau_1\eta_1$ and let $\theta \in [0,\pi]$ be the angle between the vectors $\xi$ and $\eta_1$. 
First assume that $0 \le \theta \le 2\pi/3$. Then, using the fact that 
$\cos(\theta) \ge - \frac{1}{2},\ \theta \in [0,2\pi/3]$ and the inequality 
 $|\eta_1|=|\eta| \le |\xi|/3$ (see \ref{crucialtriangle}), we have
\begin{align}  
& |\xi + \tau_1\eta_1| ^2= |\xi|^2 +\tau_1^2 |\eta_1|^2 +2\tau_1|\xi||\eta_1|\cos(\theta) \nn \\
& \ge |\xi|^2 -|\xi|^2/3=\frac{2}{3}|\xi|^2,\ \theta \in [0,2\pi/3]. \label{geometry1}
\end{align}
For $\theta \in [2\pi/3 , \pi]$, we have $\frac12 \le -\cos(\theta) \le 1$. Combined with the fact
$\tau_1 \le 1$ and $|\eta_1| \le |\xi|/3$, we readily get
\begin{gather}  \label{geometry2}
 |\xi + \tau_1\eta_1| ^2 \ge |\xi+\eta_1|^2 = |\xi-\eta|^2,\ \theta \in [2\pi/3,\pi].
\end{gather}
Combining (\ref{extt}), (\ref{geometry1}), (\ref{geometry2}), and recalling that $\phi$ is bounded
and $\alpha \le 1$,
we obtain
\begin{align*}
& |II| \le Cs^\alpha\left\{\int_0^1 \int |\eta| |\estt \cf \sk f(\eta)|\frac{|\esxt \cf \dk g(\xi-\eta)|}
{|\xi-\eta|^{1-\alpha}}d\eta\, d\tau \right. \\
& \qquad \qquad \left. \frac{1}{|\xi|^{1-\alpha}}
\int_0^1 \int |\eta| |\estt \cf \sk f(\eta)||\esxt \cf \dk g(\xi-\eta)|d\eta\, d\tau \right\}\\
& III + IV.
\end{align*}
}
From (\ref{geometry}), we obtain
\begin{gather}   \label{commutest}
 |II| \le C\frac{s^{\frac{\alpha}{\kappa}}}{|\xi|^{1-\alpha}} 
\int_0^1 \int |\eta| |\estt \cf \sk f(\eta)||\esxt \cf \dk g(\xi-\eta)|d\eta\, d\tau .
\end{gather}
\comments{
To estimate $\|III\|$, 
we apply (\ref{covineq1}) with $\zeta_1=\delta_1+\eta, \zeta_2=\delta_2+(1-\alpha)$, (\ref{spectralloc})
and Minkowski's inequality and get
\begin{align}
\|III\| & \le Cs^{\frac{\alpha}{\kappa}}2^{-(\delta_1+\delta_2+\eta-\alpha)j}\int_0^1
 \|\sk e^{s^\alpha \tau^\alpha \Lambda^\alpha} f\|_{\hh^{1+\delta_1+\eta}}\|\dk \es g\|_{\hh^{\delta_2}}
\, d\tau \nn \\
& \le Cs^\alpha 2^{-(\delta_1+\delta_2+\eta-\alpha)j}  \|\sk \es f\|_{\hh^{1+\delta_1}}\|\dk \es g\|_{\hh^{\delta_2}}
\int_0^1 \frac{1}{s^\eta(1-\tau^\alpha)^{\eta/\alpha}}\, d\tau \nn \\
& \le Cs^{\alpha -\eta}2^{-(\delta_1+\delta_2+\eta-\alpha)j} \|\sk \es f\|_{\hh^{1+\delta_1}}\|\dk \es g\|_{\hh^{\delta_2}}.  \label{IIIest}
\end{align}
In the above inequalities, we have used the fact that $\int_0^1 \frac{1}{(1-\tau^\alpha)^{\eta/\alpha}}\, d\tau <\infty$, since $\alpha \le 1$ and $\eta < \alpha$.
}
Since $II$ is non-zero only for $\xi \in \cA (2^{k-2},2^{k+2})$ and $|j - k|\le 2$ (see (\ref{spectralloc})), 
we have
\begin{gather*}
\frac{1}{|\xi|^{1-\alpha}} \le C2^{-(1-\alpha)j}. 
\end{gather*}
Thus, from (\ref{commutest}), we obtain
\begin{gather}   \label{commutest1}
 |II| \le C\frac{s^{\frac{\alpha}{\kappa}}}{2^{(1-\alpha)j}} 
\int_0^1 \int |\eta| |\estt \cf \sk f(\eta)||\esxt \cf \dk g(\xi-\eta)|d\eta\, d\tau .
\end{gather}

To the inequality in (\ref{commutest1}), we apply
 (\ref{convineq1}) with $\zeta_1=\delta_1+\zeta, \zeta_2=\delta_2$, followed by
 Minkowski inequality (in order to switch $d\tau $ and $d\xi $ integrals while computing relevant norms). 
Consequently, consulting also the second inequality in (\ref{spectralloc1}), we obtain 
\begin{gather}  \label{semigpapply}
\|II\|  \le Cs^{\frac{\alpha}{\kappa}}2^{-(\delta_1+\delta_2+\zeta-\alpha)j}\int_0^1
 \|\sk e^{\lambda {\D s^\frac{\alpha}{\kappa} \tau^\alpha \Lambda^\alpha} }f\|_{\hh^{1+\delta_1+\zeta}}\|\dk \es g\|_{\hh^{\delta_2}}
\, d\tau. 
\end{gather}
Now note that 
\begin{align*}
& \|\sk e^{\lambda {\D s^{\frac{\alpha}{\kappa}}\tau^{\alpha} \Lambda^{\alpha}}} f\|_{\hh^{1+\delta_1+\zeta}}
=\|\Lambda^{\zeta} e^{-\lambda {\D s^{\frac{\alpha}{\kappa}}(1 -\tau^{\alpha}) \Lambda^{\alpha }}}\sk\es f\|_{\hh^{1+\delta_1}}
 \\
& \le \frac{C}{(s^{\frac{\alpha}{\kappa}} (1 -\tau^{\alpha}) )^{\zeta/\alpha}} \|\sk \es f\|_{\hh^{1+\delta_1}}, 
\end{align*}
where the last inequality follows from Plancherel equality and the elementary estimate
\begin{gather}  \label{elemsemigp}
\sup_{x>0}x^me^{-tx^{\kappa}} \le \frac{C(m, \kappa)}{t^{m/\kappa}}, \qquad (m\ge 0).
\end{gather}

Therefore, from (\ref{semigpapply}), we obtain
\begin{align}
&  \|II\|  \le Cs^{\frac{\alpha}{\kappa}} 2^{-(\delta_1+\delta_2+\zeta-\alpha)j}  \|\sk \es f\|_{\hh^{1+\delta_1}}\|\dk \es g\|_{\hh^{\delta_2}}
\int_0^1 \frac{1}{s^{\frac{\zeta}{\kappa}}(1-\tau^\alpha)^{\zeta/\alpha}}\, d\tau \nn \\
& \le Cs^{\frac{(\alpha -\zeta)}{\kappa}}2^{-(\delta_1+\delta_2+\zeta-\alpha)j} \|\sk \es f\|_{\hh^{1+\delta_1}}\|\dk \es g\|_{\hh^{\delta_2}} \nn \\
& \le  Cs^{\frac{(\alpha -\zeta)}{\kappa}}2^{-(\delta_1+\delta_2+\eta-\alpha)j} c_j\| \es f\|_{\hh^{1+\delta_1}}\|\es g\|_{\hh^{\delta_2}}, \label{IIest}
\end{align}
where $c_j$ is as defined in (\ref{Iest}).
To obtain (\ref{IIest}), we also used the fact that
\[
\int_0^1 \frac{1}{(1-\tau^\alpha)^{\zeta/\alpha}}\, d\tau <\infty,\ (\mbox{since}\ \alpha \le 1\ 
\mbox{and}\  \zeta < \alpha).
\]
\comments{
\begin{gather}  \label{IVest}
\|IV\| \le  Cs^{\alpha-\eta} 2^{-(\delta_1+\delta_2+\eta-\alpha)j} \|\sk \es f\|_{\hh^{1+\delta_1}}\|\dk \es g\|_{\hh^{\delta_2}}.
\end{gather}
}
Putting together (\ref{Iest}) and (\ref{IIest}), we obtain (\ref{firstcommutest}).
\end{proof}

\noindent
{\em Proof of Theorem \ref{thm:commutest}:}
 Note that by (\ref{paraprod}),
\begin{align}
\lefteqn{[f,\es \dj]g } \nn \\
&\quad = [T_f,\es\dj]g + T_{\es\dj g}f - \es \dj (T_gf) \nn \\
& \quad +R(f,\es\dj g) - \es\dj R(f,g).  \label{gevreyparaprod}
\end{align} 
Concerning the first term $ [T_f,\es\dj]g $,
the inequality stated in (\ref{maincommutest}) follows immediately from 
 (\ref{firstcommutest}).
We will now estimate the remaining terms on the right hand side of (\ref{gevreyparaprod}).
Observing that $\es \dj=\dj \es$ and using (\ref{spectralloc}),
\comments{ replacing $g$ by $\es g$
 and proceeding as in Miura (page 147, top part),} we obtain 
\begin{align} 
& \|T_{\es \dj g}f\| =\|\sum_{k:k\ge j+2}(\sk \dj \es g)(\dk f)\| \nn \\
& \le \|\dj \es g\|\sum_{k:k\ge j+2} \|\cf \dk f\|_{L^1} \le 
C \|\dj \es g\|\sum_{k:k\ge j+2} \|\dk f\|_{\hh^1}
\label{step1} \\
& \le C2^{-j\delta_2}\|\dj \es g\|_{\hh^{\delta_2}}\sum_{k:k\ge j+2} 2^{-k\delta_1}\|\dk f\|_{\hh^{1+\delta_1}} \label{step2} \\
& \le C2^{-j\delta_2}\|\dj \es g\|_{\hh^{\delta_2}}\left(\sum_{k:k \ge j+2} 2^{-2k\delta_1}\right)^{1/2}
\left(\sum_{k:k\ge j+2} \|\dk f\|_{\hh^{1+\delta_1}}^2\right)^{1/2} \nn \\
& \le Cc_j2^{-(\delta_1+\delta_2)j}\|\es  g\|_{\hh^{\delta_2}}\|f\|_{\hh^{1+\delta_1}},\label{2ndterm}
\end{align}
where $c_j$ is as in (\ref{Iest}) and in order to obtain (\ref{step1}) and (\ref{step2}), 
we successively used Young's convolution inequality, (\ref{bernsteinfourier})  and (\ref{bernstein1}).
Proceeding in a similar manner, with $c_j$ as in (\ref{Iest}), we obtain
\begin{gather} \label{4thterm}
\|R(f,\es \dj gf\| \le C2^{-(\delta_1+\delta_2)j}c_j\|\es  g\|_{\hh^{\delta_2}}\|f\|_{\hh^{1+\delta_1}}.
\end{gather}

We will now estimate $\|\es \dj ( T_gf)\|$. Due to (\ref{spectralloc}), we have
\begin{gather*}
\es \dj (T_gf) = \sum_{k:|k-j| \le 2} \es \dj \sk g \dk f=\sum_{k:|k-j| \le 2} \dj \es \sk g \dk f.
\end{gather*}
We have
\begin{align*}
& \| \dj \es (\sk g \dk f)\| \le C2^{-(\delta_1+\delta_2-1)j}\|\es (\sk g \dk f)\|_{\hh^{\delta_1+\delta_2-1}}\\
& \le  C2^{-(\delta_1+\delta_2-1)j}\|\es \dk f\|_{\hh^{\delta_1}}\|\es \sk g\|_{\hh^{\delta_2}}\\
&\le C2^{-(\delta_1+\delta_2)j}\|\es \dk f\|_{\hh^{1+\delta_1}}\|\es \sk g\|_{\hh^{\delta_2}} \\
& \le Cc_j2^{-(\delta_1+\delta_2)j}\|\es  f\|_{\hh^{1+\delta_1}}\|\es  g\|_{\hh^{\delta_2}},\ 
c_j = \frac{\left(\sum_{k=j-2}^{j+2} \|\es \dk f\|^2_{\hh^{1+\delta_1}}\right)^{1/2}}{\|\es  f\|_{\hh^{1+\delta_1}}},
\end{align*}
where the first inequality in the above line is obtained using (\ref{bernstein1}), the second using 
(\ref{convineq2}) and the third again by (\ref{bernstein1}). Additionally, we have also used the fact
that $k \in [j-2,j+2] \cap {\mathbb N}$.

Finally, we will prove that there exists $\{c_j\}_{-\infty}^{\infty}, c_j \ge 0$ with $\sum c_j^2 \le 1$ such that
\begin{gather}  \label{remainder}
\|\es \dj R(f,g)\| \le C2^{-(\delta_1+\delta_2)j}c_j \|\es f\|_{\hh^{1+\delta_1}}
\|\|\es g\|_{\hh^{\delta_2}}, \delta_1+\delta_2 >0.
\end{gather}
\comments{Denote $\tdk = \sum\limits_{i:|i-k| \le 3 }\Delta_i$. With this notation, $R(f,g) = \sum_k \tdk f \dk g$.
Note also that by spectral localization,
$\dj(\tdk f \dk g)= 0, k < j-6$. Thus,}
From (\ref{spectralloc2}) we have
\begin{align}
& \|\es \dj R(f,g)\| \le 2^{-(\delta_1+\delta_2)j}\sum_{k \ge j-6}2^{(\delta_1+\delta_2)j}
\|\es (\tdk f \dk g)\| \nn \\
& \le 2^{-(\delta_1+\delta_2)j}\sum_{k \ge j-6}2^{(\delta_1+\delta_2)j} \|\cf \es \tdk f\|_{L^1}
\|\es \dk g\| \label{applygevreyconvolution}  \\
& \le 2^{-(\delta_1+\delta_2)j}\sum_{k \ge j-6}2^{(\delta_1+\delta_2)j} 2^k\|\es \tdk f\|
\|\es \dk g\| \label{applybernsteinfourier} \\
& \le 2^{-(\delta_1+\delta_2)j}\sum_{k \ge j-6}2^{(\delta_1+\delta_2)(j-k)}2^{(\delta_1+1)k}\|\es \tdk f\|
\, 2^{\delta_2k}\|\es \dk g\|, \label{convterm}
\end{align}
where to obtain (\ref{applygevreyconvolution} ), we used (\ref{gevreyconvolution}), while to obtain
(\ref{applybernsteinfourier}), we used (\ref{bernsteinfourier}).
Let $(a_k)_{k\in\Z}$ and $(b_k)_{k\in\Z}$ be sequences defined by
\begin{gather*}
a_k = 2^{(\delta_1+1)k}\|\es \tdk f\|
2^{\delta_2k}\|\es \dk g\|, b_k = {\chi}_{[-6,\infty)}(k)2^{-(\delta_1+\delta_2)k}.
\end{gather*}
Applying Cauchy-Schwartz and the second inequality in (\ref{bernstein2}),  we have
\begin{gather}   \label{aknorm}
\|(a_k)_{k\in \Z}\|_{\ell_1} \le C\|\es f\|_{\hh^{1+\delta_1}}\|\es g\|_{\hh^{\delta_2}}.
\end{gather}
 Define 
\begin{align*}
c_j& = \frac{1}{C\|\es f\|_{\hh^{1+\delta_1}}\|\es g\|_{\hh^{\delta_2}}}\sum_{k\in \Z}b_{j-k}a_k\\
& =\frac{1}{C\|\es f\|_{\hh^{1+\delta_1}}\|\es g\|_{\hh^{\delta_2}}}\sum_{k \ge j-6}2^{(\delta_1+\delta_2)(j-k)}2^{(\delta_1+1)k}\|\es \tdk f\|
2^{\delta_2k}\|\es \dk g\|,
\end{align*}
where $C$ is as in (\ref{aknorm}). Now using (\ref{aknorm}), the fact that $\|(b_k)\|_{\ell_2} < \infty $
 ( since $\delta_1+\delta_2 >0$) and Young's convolution inequality for sequences, 
we get that $\sum c_j^2 \le 1$. Using this fact,  we immediately obtain (\ref{remainder}) from (\ref{convterm}). 

\section{Proof of Main Result}
We will need the following lemma, the proof of which follows that of Lemma 8 in \cite{oliver2000remark}.
\begin{lemma}
Let $\alpha < \kappa $ and
$f$ is such that $f \in \hh^{\alpha/2}$ and $\es f \in \hh^{\kappa/2}$. Then $\es f \in \hh^{\alpha/2}$
and we have the  estimate 
\begin{gather}  \label{titiest}
\|\es f\|_{\hh^{\alpha/2}}^2 \le e\|f\|_{\hh^{\alpha/2}}^2 + 
(2\lambda)^{\frac{\kappa}{\alpha}-1}s^{1-\frac{\alpha}{\kappa}}\|\es f\|_{\hh^{\kappa/2}}^2.
\end{gather}
\end{lemma}
\begin{proof}
From the Plancherel equality, we have
\begin{gather}  \label{lemplancherel}
\|\es f\|_{\hh^{\alpha/2}}^2 = \int |\xi|^\alpha 
e^{2\lambda s^{\frac{\alpha}{\kappa}}|\xi|^\alpha} |(\cf f)(\xi)|^2\, d\xi.
\end{gather}
Moreover, for all $x \ge 0$ and $m>0$, we have the inequality $e^x \le e+x^me^x$. This is due to the fact that $e^x \le e$ for $x \in [0,1]$ and $e^x \le x^me^x$ for $x \ge 1$. Applying this to (\ref{lemplancherel})
with $x = 2\lambda s^{\frac{\alpha}{\kappa}}|\xi|^\alpha$ and $m= \frac{\kappa}{\alpha}-1$, 
we obtain the desired conclusion.
\end{proof}
We will also need an estimate for the linear term given in the lemma below.
\begin{lemma}  \label{lem:linest}
Let $\theta_0 \in \hh^{2-\kappa}$ and $\beta>0$. Denote
\begin{gather}   \label{caloric}
 \|\theta_0\|_{E_T} = \sup_{0<s\le T} s^{\frac{\beta}{\kappa}}\|\es e^{-s\Lambda^\kappa}\theta_0\|_{\hh^{2-\kappa+\beta}}.
\end{gather}
In this setting, with a constant $C$ independent of $T$ and $\theta_0$, we have
\begin{gather}  \label{linest}
\|\theta_0\|_{E_T} \le C\|\theta_0\|_{\hh^{2-\kappa}}\ \mbox{and}\
\lim_{T \ra 0} \|\theta_0\|_{E_T} =0.
\end{gather}
\end{lemma}
\begin{proof}
Observe that
\begin{align}
&\|\es e^{-s{\Lambda}^\kappa}\theta_0\|_{\hh^{2-\kappa+\beta}}^2
= \int \left( |\xi|^{2-\kappa+\beta}\esx e^{-s|\xi|^\kappa}|(\cf \theta_0)(\xi)|\right)^2\, d\xi   \nn \\
& = \int \left( |\xi|^{2-\kappa+\beta} e^{\lambda s^{\frac{\alpha}{\kappa}}|\xi|^{\alpha}-\frac{s}{2}|\xi|^\kappa}e^{-\frac{s}{2}|\xi|^\kappa}|(\cf \theta_0)(\xi)|\right)^2\, d\xi.   \label{tolinest}
\end{align}
Now observe that 
\begin{gather} \label{intmedlin}
 \sup_{s\ge 0,\xi \in \R^2}e^{\lambda s^{\frac{\alpha}{\kappa}}|\xi|^{\alpha}-\frac{s}{2}|\xi|^\kappa}
=\sup_{s\ge 0,\xi \in \R^2}e^{\lambda (s^{\frac{1}{\kappa}}|\xi|)^{\alpha}-\frac{1}{2}(s^{\frac{1}{\kappa}}|\xi|)^\kappa} \le C(\lambda, \kappa , \alpha),
\end{gather}
since, for $\alpha < \kappa$, the function $f(x)=\lambda x^\alpha -\frac 12 x^\kappa \le C(\lambda , \alpha , \kappa)$ for all $x>0$.
Applying (\ref{intmedlin}) and (\ref{elemsemigp}) to (\ref{tolinest}), we obtain the first inequality in 
(\ref{linest}). In case $\theta_0' \in \hh^{2-\kappa+\beta}$, a  similar calculation using  (\ref{intmedlin})
 yields
\begin{gather}  \label{smoothdata}
\|\es e^{-s\Lambda^\kappa}\theta_0'\|_{\hh^{2-\kappa+\beta}} \le C\|\theta_0'\|_{\hh^{2-\kappa+\beta}}.
\end{gather}
Given $\epsilon >0$, we can choose $\theta_0'$ such that 
\begin{gather}  \label{approximate}
\|\theta_0'-\theta_0\|_{\hh^{2-\kappa}} \le \epsilon\ \mbox{and}\ \theta_0' \in \hh^{2-\kappa+\beta}.
\end{gather}
From the first inequality in (\ref{linest}), (\ref{smoothdata}) and (\ref{approximate}), the second assertion in (\ref{linest}) immeditely follows.

\end{proof}
Before embarking on the proof of Theorem \ref{thm:main}, we note that
\begin{gather*}
\|u\|_{\hh^m} \simeq \|\theta\|_{\hh^m}, \ m \in \R,
\end{gather*}
since they are related by the Riesz transform as given in (\ref{QG}).\\[2pt]

{\em  Proof of Theorem \ref{thm:main}.}
As is customary,
we consider the following approximate sequence of solutions:
\begin{gather}
\left.
\begin{array}{l}
\partial_t \theta^{(n+1)} + \Lambda^\kappa \theta^{n+1} + u^{(n)}\cdot \nabla \theta^{(n+1)} =0,
\ \theta^{(n+1)}|_{t=0}=\theta_0, \\
u^{(n)} = (-R_2\theta^{(n)},R_1\theta^{(n)}), n=0, 1, \cdots ,
\end{array}
\right\}
\end{gather}
with the convention that $\theta^{(-1)} \equiv 0$ and $u^{(-1)}=0$. Denote 
\begin{gather}  \label{tildes}
\tn(s) = \es \theta^{(n)}(s), \un(s)=\es u^{(n)}(s)=(-R_2\tn,R_1\tn);
\end{gather}
 the very last equality above is due to the fact that $R_i,\, i=1,2,$ commute with $\es $.

Due to Theorem 1 and Corollary 1 in \cite{miura}, provided either $\|\theta_0\|_{\h^{2-\kappa}}$ or $T$ is sufficiently small,
the sequence $\{\theta^{(n)}\}$ converges (in $\h^{2-\kappa}$) to a solution $\theta $ of (\ref{QG})
which additionally belongs to $C([0,T];\h^{2-\kappa})$. Moreover, for all $n$, 
$\{\theta^{(n)} \}$ satisfies 
\begin{gather}  \label{solprop}
\|\theta^{(n)}\|_{E_T}=\sup_{0<s\le T}
 s^{\frac{\beta}{\kappa}}\|\theta^{(n)}(s)\|_{\hh^{2-\kappa+\beta}} \le C\|\theta_0\|_{\hh^{2-\kappa}}
\ \mbox{and}\ \lim_{T\ra 0}\sup_n \|\theta^{(n)}\|_{E_T}=0;
\end{gather}
the constant $C$ above is independent of $T$ and $\theta_0$.
Thus,
in order to prove the Theorem \ref{thm:main},
it will be sufficient to demonstrate {\em a priori} esimates, i.e., to obtain bounds on
 $\|\tn(\cdot)\|_{E_T}$, independent of $n$.

For the remainder of the proof, we choose, and fix, the parameters $\beta, \zeta$ by
\begin{gather}   \label{paramchoice}
0< \beta < \min \left\{ \frac{\kappa}{2}, 2(\kappa-\alpha), \alpha\right\}
\ \mbox{and}\ \zeta = \alpha - \frac{\beta}{2}.
\end{gather}
Using respectively the facts that $u^{(n)}$ is divergence free and the operators
$\es \dj$ and $ \es \sj$ are Fourier multipliers (and hence commute with $\nabla$),
 we have
\begin{align}  
& \es \dj \nabla \theta^{(n)}(s) = \nabla \dj \tn(s),\ 
\es \sj \nabla \theta^{(n)} (s) = \nabla \sj \tn (s)\ \mbox{and}\ \nn \\
& \langle u^{(n)} \cdot \nabla \dj \tno , \dj \tno \rangle =0, \label{misc}
\end{align}
where $\tn, \un$ are as in \eqref{tildes}.
From (\ref{QG}) and (\ref{misc}), taking $L^2$-inner product, we readily obtain
\begin{align}
& \frac12 \frac{d}{ds}\|\dj \tno\|^2 + \|\Lambda^{\kappa/2}\dj \tno\|^2 \nn \\
& = \lambda \frac{\alpha}{\kappa} s^{\frac{\alpha}{\kappa}-1}\|\Lambda^{\alpha/2}\dj \tno\|^2 +
\l [u^{(n)},\es \dj]\nabla\theta^{(n+1)},\dj \tno\r  \nn \\
& \le C(\alpha, \kappa)\lambda^{\frac{\kappa}{\alpha}}\|\Lambda^{\kappa/2}\dj \tno\|^2
+\lambda s^{\frac{\alpha}{\kappa}-1} C(\alpha, \kappa)\|\Lambda^{\alpha/2}\dj \theta^{(n+1)}\|^2 \nn \\
& \qquad + \l [u^{(n)},\es \dj]\nabla\theta^{(n+1)},\dj \tno\r , \label{titi}
\end{align}
where to obtain the inequality (\ref{titi}), we applied \eqref{titiest}. Since $\alpha < \kappa $ and 
$C(\alpha, \kappa)$ is independent of $\lambda $, we can choose (and henceforth, fix) $ \lambda $ small enough so that $C(\alpha, \kappa)\lambda^{\frac{\kappa}{\alpha}} < \frac12$. 
Note that due to \eqref{paramchoice},   the parameters  $\beta$ and $\zeta$ satisfy the conditions
\begin{gather}  \label{forcommutest}
\min\{\beta,\zeta\}>0, \beta + \zeta < \kappa\ \mbox{and}\ \zeta < \alpha.
\end{gather}
We can now apply Theorem \ref{thm:commutest} to the commutator term on the right hand side of inequality
(\ref{titi}) with 
\begin{gather}
\delta_1 = 1-\kappa + \beta, \delta_2 = 1-\kappa+\beta, f = u^{(n)}, g=\nabla \theta^{(n+1)} ,
\end{gather}
and Bernstein's inequality (\ref{bernstein1}) to the term $ \|\Lambda^{\kappa/2}\dj \tno\|^2$ on the left hand side of (\ref{titi}), to obtain
\begin{align*}
& \ds \|\dj \tno\|^2 + C_12^{\kappa j}\|\dj \tno\|^2 \\
& \le C\left\{2^{\alpha j}s^{\frac{\alpha}{\kappa}-1} \|\dj \theta^{(n+1)}\|^2 \right. \\
& \quad \left. + c_j\left(2^{-(2-2\kappa + 2\beta)j}+s^{\frac{(\alpha-\zeta)}{\kappa}}2^{-(2-2\kappa + 2\beta+\zeta-\alpha)j}\right)
\| \tn\|_{\hh^{2-\kappa+\beta}}\| \tno\|_{\hh^{2-\kappa+\beta}}\|\dj \tno\|\right\}.
\end{align*}
Now divide both sides by $\|\dj \tno\|$ and recall that 
${\D \frac{\|\dj \theta^{(n+1)}\|}{\|\dj \tno\|}} \le 1$. This yields 
\begin{align*}
& \ds \|\dj \tno\| + C_12^{\kappa j}\|\dj \tno\| \\
& \le C\left\{s^{\frac{\alpha}{\kappa}-1} 2^{\alpha j}\|\dj \theta^{(n+1)}\|  \right. \\
& \quad \left. + c_j\left(2^{-(2-2\kappa + 2\beta)j}+s^{\frac{(\alpha-\zeta)}{\kappa}}2^{-(2-2\kappa + 2\beta+\zeta-\alpha)j}\right)
\| \tn\|_{\hh^{2-\kappa+\beta}}\| \tno\|_{\hh^{2-\kappa+\beta}}\right\}.
\end{align*}
The variation of parameters formula, and the fact that $\tno (0)=\theta^{(n+1)}(0)=\theta_0$, now yield
\begin{align}
& \|\dj \tno(t)\| \le e^{-C_12^{\kappa j}t}\|\dj \theta_0\| 
+ C\int_0^t s^{\frac{\alpha}{\kappa}-1} 2^{\alpha j}e^{-C_12^{\kappa j}(t-s)}\|\dj \theta^{(n+1)}(s)\| \, ds 
\nn \\
& + 
 C\int_0^t  c_j2^{-(2-2\kappa+2\beta) j}e^{-C_12^{\kappa j}(t-s)}\|\tn (s)\|_{\hh^{2-\kappa+\beta}}
\|\tno (s)\|_{\hh^{2-\kappa+\beta}}\, ds
\nn \\
& + C\int_0^t c_js^{\frac{(\alpha-\zeta)}{\kappa}} 2^{-(2-2\kappa+2\beta + \zeta - \alpha) j}e^{-C_12^{\kappa j}(t-s)}\|\tn (s)\|_{\hh^{2-\kappa+\beta}}
\|\tno (s)\|_{\hh^{2-\kappa+\beta}}\, ds. \label{varparam}
\end{align}
Multiply both sides of the inequality (\ref{varparam}) by $2^{(2-\kappa+\beta)j}$ and apply
(\ref{elemsemigp}). Subsequently, take the $\ell_2$-norm of the resulting sequence
and apply Minkowski inequality (to interchange $ds$ and $\sum_j$).
\comments{ Denote 
\begin{gather*}
\|\T\|_{E_T} = \sup_{0<s\le T}s^{\frac{\beta}{\kappa}}\|\T(s)\|_{\hh^{2-\kappa+\beta}}.
\end{gather*}
}
Consequently, from (\ref{varparam}), 
 the first relation in (\ref{solprop}) and the fact that $\sum c_j^2 \le 1$,  
we obtain for all $t>0$ the estimate
\begin{align}
&\|\tno (t)\|_{\hh^{2-\kappa+\beta}}
\le \widetilde{C}_1 \frac{\|\theta_0\|_{\hh^{2-\kappa}}}{t^{\frac{\beta}{\kappa}}} + 
\widetilde{C}_2\|\theta_0\|_{\hh^{2-\kappa}}\int_0^t \frac{ds}{(t-s)^{\frac{\alpha}{\kappa}}s^{1+\frac{\beta -\alpha}{\kappa}}} \nn \\
& + \widetilde{C}_3\|\tno\|_{E_T}\|\tn\|_{E_T} \left\{
\int_0^t \frac{ds}{s^{2\frac{\beta}{\kappa}}(t-s)^{1 - \frac{\beta}{\kappa}}}
+  \int_0^t \frac{ds}{s^{\frac{2\beta-\alpha+\zeta}{\kappa}}
(t-s)^{\frac{\alpha +\kappa-\zeta-\beta}{\kappa}}}\right\} , \label{penultimate} 
\end{align}
where the constants $\widetilde{C}_i$ above are independent of $n, T, t$ and $\theta_0$
as well as the sequences $\{\theta^{(n)}\}$ and $\{\tn\}$.
The integrals on the right hand side of \eqref{penultimate} are finite because 
$\alpha < \kappa$, and due to \eqref{paramchoice}, the parameters
$\beta, \zeta$ satisfy
\begin{gather}  \label{finiteintegral}
\beta < \min\{\alpha, \frac{\kappa}{2}\}, \beta < \frac{\kappa}{2}- \frac{\zeta-\alpha}{2}\ \mbox{and}\ 
\alpha < \zeta + \beta.
\end{gather}
From (\ref{penultimate}), we easily obtain
\begin{align} 
&\|\tno(\cdot)\|_{E_T}=\sup_{0<t <T}t^{\frac{\beta}{\kappa}}\|\tno (t)\|_{\hh^{2-\kappa+\beta}} \nn \\
& \le \widetilde{C}_1\|\theta_0\|_{\hh^{2-\kappa}}+ \widetilde{C}_4\|\theta_0\|_{\hh^{2-\kappa}} + 
\widetilde{C}_5\|\tno\|_{E_T}\|\tn\|_{E_T}, \label{ultimate}
\end{align}
where
\begin{align*}
& \widetilde{C}_4=t^{\frac{\beta}{\kappa}}\tilde{C}_2\int_0^t \frac{ds}{(t-s)^{\frac{\alpha}{\kappa}}s^{1+\frac{\beta -\alpha}{\kappa}}}
\quad \mbox{and}\ \\
&\widetilde{C}_5 = t^{\frac{\beta}{\kappa}}\widetilde{C}_3\left\{\int_0^t \frac{ds}{s^{2\frac{\beta}{\kappa}}(t-s)^{1 - \frac{\beta}{\kappa}}}
+  \int_0^t \frac{ds}{s^{\frac{2\beta-\alpha+\eta}{\kappa}}
(t-s)^{\frac{\alpha +\kappa-\eta-\beta}{\kappa}}}\right\}.
\end{align*}
The integrals in the definition of $\widetilde{C}_4$ and $\widetilde{C}_5$ above are finite due to (\ref{finiteintegral}).

In a similar manner, following the derivation of (\ref{ultimate}) and using (\ref{linest}),  we can also obtain
\begin{gather} \label{ultimate1}
\|\tno\|_{E_T} \le \widetilde{C}_1\|\theta_0\|_{E_T}
+ \widetilde{C}_4\|\theta_0\|_{E_T}+\widetilde{C}_5\|\tno\|_{E_T}\|\tn\|_{E_T},
\end{gather}
where $\|\theta_0\|_{E_T}$ is as defined in (\ref{caloric}).
Assume that $\widetilde{C}_5\|\tn\|_{E_T} < 1/2$. From (\ref{ultimate}), we readily obtain,
\begin{gather}  \label{finalbd}
\|\tno\|_{E_T} \le 2(\widetilde{C}_1+\widetilde{C}_4)\|\theta_0\|_{\hh^{2-\kappa}}.
\end{gather}
Provided 
\begin{gather*}
2\widetilde{C}_5(C_1+C_4)\|\theta_0\|_{\hh^{2-\kappa}} \le \frac 12, 
\end{gather*}
we see that
$\widetilde{C}_5\|\tno\|_{E} \le \frac 12$ and by induction, (\ref{finalbd}) holds for all $n$. In case
$\theta_0 \in \hh^{2-\kappa}$ is arbitrary, for sufficiently small $T$,
we can similarly derive uniform (in $n$) bound on $\|\tno\|_{E_T}$ from  (\ref{ultimate1}), by applying (\ref{linest}) and the second relation in (\ref{solprop}).

\comments{For the necessary application of the inequalities and the integrability of the above integrals, we must have
the following constraints on the parameters.
\begin{align*}
& (i)\, \beta + \eta < \kappa \ (ii)\, \eta < \alpha  \quad \mbox{(commutator inequality application restrictions)}\\
&(iii)\, \beta < \frac{\kappa}{2}\ (iv)\, \beta < \frac{\kappa}{2}- \frac{\eta-\alpha}{2}\ \mbox{and}\ (v)\, 
\alpha < \eta + \beta.
\quad \mbox{(finiteness of integral in (\ref{ultimate}))}.
\end{align*}
}
In order that the parameters satisfy (\ref{forcommutest}) and (\ref{finiteintegral}),
 it is sufficient that the conditions
\begin{gather*}
\alpha < \eta+\beta < \kappa, \eta <\alpha \ \mbox{and}\ \beta < \frac{\kappa}{2}
\end{gather*}
hold. As long as $\alpha < \kappa$, simply take $\eta = \alpha - \frac{\beta}{2}$ with
 $\beta < \frac{\kappa}{2}$. All conditions are met and we finish the proof.

 
 \end{document}